\documentclass[12pt]{amsart}
\usepackage{amscd,amsmath,amsthm,amssymb,graphics}
\usepackage{pstcol,pst-plot,pst-3d}
\usepackage{lmodern,pst-node}
\usepackage{multicol}
\usepackage{epic,eepic}
\usepackage{amsfonts,amssymb,amscd,amsmath,enumerate,verbatim}
\psset{unit=0.7cm,linewidth=0.8pt,arrowsize=2.5pt 4}

\newpsstyle{fatline}{linewidth=1.5pt}
\newpsstyle{fyp}{fillstyle=solid,fillcolor=verylight}
\definecolor{verylight}{gray}{0.97}
\definecolor{light}{gray}{0.9}
\definecolor{medium}{gray}{0.85}
\definecolor{dark}{gray}{0.6}

\def\frk{\frak}               

\def\pp{{\frk p}}

\def\Phi{{\frk n}}
\def\Phi{{\frk N}}
\def\MR{{\mathcal R}}

\def\MS{{\mathcal S}}
\def\MC{{\mathcal C}}

\def\opn#1#2{\def#1{\operatorname{#2}}} 
\opn\chara{char} \opn\length{\ell} \opn\pd{pd} \opn\rk{rk}
\opn\projdim{proj\,dim} \opn\injdim{inj\,dim} \opn\rank{rank}
\opn\depth{depth} \opn\grade{grade} \opn\height{height}
\opn\embdim{emb\,dim} \opn\codim{codim}

\opn\Tr{Tr} \opn\bigrank{big\,rank}
\opn\superheight{superheight}\opn\lcm{lcm}
\opn\trdeg{tr\,deg}
\opn\reg{reg} \opn\lreg{lreg} \opn\ini{in} \opn\lpd{lpd}
\opn\size{size}\opn\bigsize{bigsize}
\opn\cosize{cosize}\opn\bigcosize{bigcosize}
\opn\sdepth{sdepth}\opn\sreg{sreg}
\opn\link{link}\opn\fdepth{fdepth}\opn\cdeg{cdeg}
\opn\div{div} \opn\Div{Div} \opn\cl{cl} \opn\Cl{Cl}
\opn\Spec{Spec} \opn\Supp{Supp} \opn\supp{supp} \opn\Sing{Sing}
\opn\Ass{Ass} \opn\Min{Min}\opn\Mon{Mon} \opn\dstab{dstab} \opn\astab{astab}
\opn\Syz{Syz}
\opn\Ann{Ann} \opn\Rad{Rad} \opn\Soc{Soc}
\opn\Im{Im} \opn\Ker{Ker} \opn\Coker{Coker} \opn\Am{Am}
\opn\Hom{Hom} \opn\Tor{Tor} \opn\Ext{Ext} \opn\End{End}
\opn\Aut{Aut} \opn\id{id}

\opn\nat{nat}
\opn\pff{pf}
\opn\Pf{Pf} \opn\GL{GL} \opn\SL{SL} \opn\mod{mod} \opn\ord{ord}
\opn\Gin{Gin} \opn\Hilb{Hilb}\opn\sort{sort}
\opn\aff{aff} \opn\con{conv} \opn\relint{relint} \opn\st{st}
\opn\lk{lk} \opn\cn{cn} \opn\core{core} \opn\vol{vol}
\opn\link{link} \opn\star{star}\opn\lex{lex}\opn\cdeg{cdeg}
\opn\gr{gr}

\def\pot#1#2{#1[\kern-0.28ex[#2]\kern-0.28ex]}

\opn\dirlim{\underrightarrow{\lim}}
\opn\inivlim{\underleftarrow{\lim}}
\let\to=\rightarrow

\def\Implies{\ifmmode\Longrightarrow \else
        \unskip${}\Longrightarrow{}$\ignorespaces\fi}
\def\implies{\ifmmode\Rightarrow \else
        \unskip${}\Rightarrow{}$\ignorespaces\fi}
\def\iff{\ifmmode\Longleftrightarrow \else
        \unskip${}\Longleftrightarrow{}$\ignorespaces\fi}

\let\:=\colon
\newtheorem{Theorem}{Theorem}[section]
 \newtheorem{Lemma}[Theorem]{Lemma}
 \newtheorem{Corollary}[Theorem]{Corollary}
 \newtheorem{Proposition}[Theorem]{Proposition}
 \newtheorem{Remark}[Theorem]{Remark}

\let\epsilon\varepsilon
\let\kappa=\varkappa
\def\qed{\ifhmode\textqed\fi
      \ifmmode\ifinner\quad\qedsymbol\else\dispqed\fi\fi}
\def\textqed{\unskip\nobreak\penalty50
       \hskip2em\hbox{}\nobreak\hfil\qedsymbol
       \parfillskip=0pt \finalhyphendemerits=0}
\def\dispqed{\rlap{\qquad\qedsymbol}}

\opn\dis{dis}
\def\pnt{{\raise0.5mm\hbox{\large\bf.}}}

\opn\Lex{Lex}

\begin{document}

 \title {On the symbolic powers of binomial edge ideals}

 \author {Viviana Ene, J\"urgen Herzog}

\address{Viviana Ene, Faculty of Mathematics and Computer Science, Ovidius University, Bd.\ Mamaia 124,
 900527 Constanta, Romania}  \email{vivian@univ-ovidius.ro}

\address{J\"urgen Herzog, Fachbereich Mathematik, Universit\"at Duisburg-Essen, Campus Essen, 45117
Essen, Germany} \email{juergen.herzog@uni-essen.de}

 \begin{abstract} We show that under some conditions, if the initial ideal $\ini_<(I)$ of an ideal $I$ in a polynomial ring has the property that its symbolic and ordinary powers coincide, then the ideal $I$ shares the same property. We apply this result to prove the equality between symbolic and ordinary powers for binomial edge ideals with quadratic Gr\"obner basis.
 \end{abstract}


 \maketitle

\section{Introduction}

Binomial edge ideals were introduced in \cite{HHHKR} and, independently, in \cite{Oh}. Let $S=K[x_1,\ldots,x_n,y_1,\ldots,y_n]$ be the polynomial ring in $2n$ variables over a field $K$ and $G$ a simple graph  on the vertex set $[n]$ with edge set $E(G).$ The binomial edge ideal of $G$ is generated by the set of  $2$-minors of the generic matrix
$X=\left(
\begin{array}{cccc}
x_1 & x_2 & \cdots & x_n\\
y_1 & y_2 & \cdots & y_n
\end{array}\right)
$ indexed by the edges of $G$. In other words,
\[
J_G=(x_iy_j-x_jy_i: i<j \text{ and }\{i,j\}\in E(G)).
\] We will often use the notation $[i,j]$ for the maximal minor $x_iy_j-x_jy_i$ of $X.$

In the last decade, several properties of binomial edge ideals have been studied. In \cite{HHHKR}, it was shown that, for every graph $G,$ the ideal $J_G$ is a radical ideal and the minimal prime ideals are characterized in terms of the combinatorics of the graph. Several articles considered the Cohen-Macaulay property of binomial edge ideals; see, for example, \cite{BMS, EHH, RR, Ri, Ri2}. A significant effort has been done for studying the resolution of binomial edge ideals. For relevant results on this topic we refer to the recent survey 
\cite{Sara} and the references therein.

In this paper, we consider symbolic powers of binomial edge ideals. The study and use of symbolic powers have been a reach topic of research in commutative algebra for more than 40 years. Symbolic powers and ordinary powers do not coincide in general. However, there are classes of homogeneous ideals in polynomial rings for which the symbolic and ordinary powers coincide. For example, if $I$ is the edge ideal of a graph, then $I^k=I^{(k)}$ for all $k\geq 1$  if and only if the graph is bipartite. More general, the facet ideal $I(\Delta)$ of a simplicial complex $\Delta$ has the property that $I(\Delta)^k=I(\Delta)^{(k)}$ for all $k\geq 1$ (equivalently, $I(\Delta)$ is normally torsion free) if and only if $\Delta$ is a Mengerian complex; see \cite[Section 10.3.4]{HH10}. The ideal of the maximal minors of a generic matrix shares the same property, that is, the symbolic and ordinary powers coincide \cite{DEP}.

To the best of our knowledge, the comparison between symbolic and ordinary powers for binomial edge ideals was considered so far only  in \cite{Oh2}. In Section 4 of this paper, Ohtani proved that if $G$ is a complete multipartite graph, then $J_G^k=J_G^{(k)}$ for all integers
$k\geq 1.$ 

In our paper we prove that, for any binomial edge ideal with quadratic Gr\"obner basis,  the symbolic and ordinary powers of $J_G$ coincide. The proof is based on the  transfer of the equality for symbolic and ordinary powers from the initial ideal to the ideal itself. 

The structure of the paper is the following. In Section~\ref{one} we survey basic results needed in the next section on symbolic powers of ideals in Noetherian rings and on binomial edge ideals and their primary decomposition. 

In Section~\ref{three} we discuss symbolic powers in connection to initial ideals. Under some specific conditions on the homogeneous ideal $I$ in a polynomial ring over a field, one may derive that if $\ini_<(I)^k=\ini_<(I)^{(k)}$ for some integer $k\geq 1,$ then $I^k=I^{(k)}$; see Lemma~\ref{inilemma}. By using this lemma and the properties of binomial edge ideals, we show in Theorem~\ref{iniconseq} that if 
$\ini_<(J_G)$ is a normally torsion-free ideal, then the symbolic and ordinary powers of $J_G$ coincide. This is the case, for example, if 
$G$ is a closed graph (Corollary~\ref{closed}) or the cycle $C_4.$ However, in general, $\ini_<(J_G)$  is not a normally torsion-free ideal.  For example, for the binomial edge ideal   of the $5$--cycle,   we have $J_{C_5}^2=J_{C_5}^{(2)}$, but $(\ini_<(J_{C_5}))^2\subsetneq (\ini_<(J_{C_5}))^{(2)}.$

\section{Preliminaries}
\label{one}

In this section we summarize basic facts about symbolic powers of ideals and binomial edge ideals. 

\subsection{Symbolic powers of ideals}

Let $I\subset R$ be an ideal in a Noetherian ring $R,$ and let $\Min(I)$ the set of the minimal prime ideals of $I.$ For an iteger $k\geq 1,$ one defines the \emph{$k^{th}$ symbolic power} of $I$ as follows: 
\[
I^{(k)}=\bigcap_{\pp\in\Min(I)}(I^kR_\pp\cap R)=\bigcap_{\pp\in\Min(I)}\ker(R\to (R/I^k)_\pp)=\]
\[=\{a\in R: \text{ for every }\pp\in \Min(I), \text{ there exists }w_\pp\not\in \pp \text{ with }w_\pp a\in I^k\}=
\]
\[=\{a\in R: \text{ there exists }w\not\in \bigcup_{\pp\in \Min(I)}\pp\text{ with } wa\in I^k\}.
\]

By the definition of the symbolic power, we have $I^k\subseteq I^{(k)}$ for $k\geq 1. $ Symbolic powers do not, in general, coincide with the ordinary powers. However, if $I$ is a complete intersection or it is the determinantal ideal generated by the maximal minors 
of a generic matrix, then it is known that $I^k= I^{(k)}$ for $k\geq 1;$ see \cite{DEP} or \cite[Corollary 2.3]{BC03}.

Let $I=Q_1\cap \cdots \cap Q_m$ an irredundant primary decomposition of $I$  with $\sqrt{Q_i}=\pp_i$ for all $i.$ If the minimal prime ideals of $I$ are $\pp_1,\ldots \pp_s,$ then 
\[I^{(k)}=Q_1^{(k)}\cap\cdots \cap Q_s^{(k)}. \]

In particular, if $I\subset R=K[x_1,\ldots,x_n]$ is a square-free monomial ideal in a polynomial ring over a field $K$, then 
\[I^{(k)}=\bigcap_{\pp\in \Min(I)} \pp^k.
\]
Moreover, $I$ is normally torsion-free (i.e. $\Ass(I^m)\subseteq \Ass(I)$ for $m\geq 1$) if and only if $I^k= I^{(k)}$ for all $k\geq 1,$ if and only if $I$ is the Stanley-Reisner ideal of a Mengerian simplicial complex; see \cite[Theorem 1.4.6, Corollary 10.3.15]{HH10}. In particular, if $G$ is a bipartite graph, then its monomial edge ideal $I(G)$ is normally torsion-free \cite[Corollary 10.3.17]{HH10}.

In what follows, we will often use the binomial expansion of symbolic powers \cite{HNTT}. Let $I\subset R$ and $J\subset R^\prime$ be two homogeneous ideals in the polynomial algebras $R,R^\prime$ in disjoint sets of variables over the same field $K$. We write $I,J$ for the extensions of these two ideals in $R\otimes_K R^\prime.$ Then, the following binomial expansion holds.

\begin{Theorem}\cite[Theorem 3.4]{HNTT} In the above settings,  \[ (I+J)^{(n)}=\sum_{i+j=n}I^{(i)}J^{(j)}.\]
\end{Theorem}

Moreover, we have the following criterion for the equality of the symbolic and ordinary powers.

\begin{Corollary}\cite[Corollary 3.5]{HNTT} \label{corh} In the above settings, assume that $I^t\neq I^{t+1}$ and $J^t\neq J^{t+1}$ for $t\leq n-1.$ Then 
$(I+J)^{(n)}=(I+J)^n$ if and only if $I^{(t)}=I^t$ and $J^{(t)}=J^t$ for every $t\leq n.$
\end{Corollary}

\subsection{Binomial edge ideals}

Let $G$ be a simple graph on the vertex set $[n]$ with edge set $E(G)$ and let $S$ be the polynomial ring $K[x_1,\ldots,x_n,y_1,\ldots,y_n]$  in $2n$ variables over a 
field $K.$ The binomial edge ideal $J_G\subset S$ associated with $G$ is
\[
J_G=(f_{ij}: i<j,  \{i,j\}\in E(G)),
\] where $f_{ij}=x_iy_j-x_jy_i$ for $1\leq i<j\leq n.$ Note that $f_{ij}$ are exactly the maximal minors of the $2\times n$ generic matrix
$X=\left(
\begin{array}{cccc}
x_1 & x_2 & \cdots & x_n\\
y_1 & y_2 & \cdots & y_n
\end{array}\right).
$ We will use the notation $[i,j]$ for the $2$- minor of $X$ determined by the columns $i$ and $j.$

We consider the polynomial ring $S$ endowed with the lexicographic order induced by the natural order of the variables, and $\ini_<(J_G)$ denotes the initial ideal of $J_G$ with respect to this monomial order. By \cite[Corollary 2.2]{HHHKR}, $J_G$ is a radical ideal. Its minimal prime ideals may be characterized in terms of the combinatorics of the graph $G.$ We introduce the following notation. 
Let $\MS\subset [n]$ be a (possible empty) subset of $[n]$, and let $G_1,\ldots,G_{c(\MS)}$ be the connected components of $G_{[n]\setminus \MS}$ where 
$G_{[n]\setminus \MS}$ is the induced subgraph of $G$ on the vertex set $[n]\setminus \MS.$ For $1\leq i\leq c(\MS),$ let $\tilde{G}_i$ be the complete 
graph on the vertex set $V(G_i).$ Let \[P_{\MS}(G)=(\{x_i,y_i\}_{i\in \MS}) +J_{\tilde{G_1}}+\cdots +J_{\tilde{G}_{c(\MS)}}.\]

Then $P_{\MS}(G)$ is a prime ideal. Since the symbolic powers of an  ideal of maximal minors of a generic matrix coincide with the ordinary powers, and by using 
Corollary~\ref{corh}, we get 
\begin{equation}\label{eqprime}
P_{\MS}(G)^{(k)}=P_{\MS}(G)^k \text{ for } k\geq 1.
\end{equation} 

By \cite[Theorem 3.2]{HHHKR}, $J_G=\bigcap_{\MS\subset [n]}P_{\MS}(G).$ In particular, the minimal primes of $J_G$ are among the prime ideals $P_{\MS}(G)$ with $\MS\subset [n].$
The following proposition characterizes the sets $\MS$ for which the prime ideal $P_{\MS}(G)$ is minimal.

\begin{Proposition}\label{cpset}\cite[Corollary 3.9]{HHHKR}
$P_{\MS}(G)$ is a minimal prime of $J_G$ if and only if either $\MS=\emptyset$ or $\MS$ is non-empty and for each $i\in \MS,$ $c(\MS\setminus\{i\})<c(\MS)$.
\end{Proposition}

In combinatorial terminology, for a connected graph $G$, $P_{\MS}(G)$ is a minimal prime ideal of $J_G$ if and only if $\MS$ is empty or $\MS$ is non-empty and  is a \emph{cut-point set} of $G,$ that is, $i$ is a cut point of the restriction $G_{([n]\setminus\MS)\cup\{i\}}$ for every $i\in \MS.$ Let $\MC(G)$ be the set of all sets $\MS\subset [n]$ such that $P_{\MS}(G)\in \Min(J_G).$

Let us also mention that, by  \cite[Theorem 3.1]{CDeG} and \cite[Corollary 2.12]{CDeG}, we have 
\begin{equation}\label{intersectini}
\ini_<(J_G)=\bigcap_{\MS \in \MC(G)} \ini_< P_{\MS}(G).
\end{equation}

\begin{Remark} {\em The cited results of \cite{CDeG} require that $K$ is algebraically closed. However, in our case, we may remove this condition on the field $K.$ Indeed, neither the Gr\"obner basis of $J_G$ nor the primary decomposition of $J_G$  depend on the field $K,$ thus we may extend the field $K$ to its algebraic closure $\bar{K}.$}
\end{Remark}

\medskip

When we study symbolic powers of binomial edge ideals, we may reduce to connected graphs. Let $G=G_1\cup \cdots \cup G_c$ where $G_1,\ldots,G_c$ are the connected components of $G$ and $J_G\subset S$ the binomial edge ideal of $G.$ Then we may write
\[J_G=J_{G_1}+\cdots +J_{G_c}
\] where $J_{G_i}\subset S_i=K[{x_j,y_j: j\in V(G_i)}]$ for $1\leq i\leq c.$ In the above equality, we used the notation  $J_{G_i}$ for 
the extension of $J_{G_i}$ in $S$ as well.

\begin{Proposition}\label{connected}
In the above settings, we have $J_G^k=J_G^{(k)}$ for every $k\geq 1$ if and only if $J_{G_i}^k=J_{G_i}^{(k)}$ for every $k\geq 1.$
\end{Proposition}

\begin{proof}
The equivalence is a direct consequence of Corollary~\ref{corh}.
\end{proof}

\section{Symbolic powers and initial ideals}
\label{three}

In this section we discuss the transfer of the equality between symbolic and ordinary powers from the initial ideal to the ideal itself. 

Let $R=K[x_1,\ldots,x_n]$ be the polynomial ring over the field $K$ and $I\subset R$ a homogeneous ideal. We assume that there exists a monomial order $<$ on $R$ such that $\ini_<(I)$ is a square-free monomial ideal. In particular, it follows that $I$ is a radical ideal. Let 
$\Min(I)=\{\pp_1,\ldots,\pp_s\}.$ Then $I=\bigcap_{i=1}^s \pp_i.$

\begin{Lemma}\label{inilemma}
In the above settings, we assume that the following conditions are fulfilled: 
\begin{itemize}
\item [(i)] $\ini_<(I)=\bigcap_{i=1}^s \ini_<(\pp_i);$
\item [(ii)] For an integer $t\geq 1$ we have: 
\begin{itemize}
	\item [(a)] $\pp_i^{(t)}=\pp_i^t$ for $1\leq i\leq s;$
	\item [(b)] $\ini_<(\pp_i^t)=(\ini_<(\pp_i))^t$ for $1\leq i\leq s;$
	\item [(c)] $(\ini_<(I))^{(t)}=(\ini_<(I))^t.$
\end{itemize}
\end{itemize}
Then $I^{(t)}=I^t.$
\end{Lemma}

\begin{proof}
In our hypothesis, we obtain:
\[\ini_<(I^t)\supseteq (\ini_<(I))^t=(\ini_<(I))^{(t)}=\bigcap_{i=1}^s(\ini_<(\pp_i))^{(t)}\supseteq 
\bigcap_{i=1}^s(\ini_<(\pp_i))^{t}=\bigcap_{i=1}^s \ini_<(\pp_i^t)\supseteq\]
\[\supseteq \ini_<(\bigcap_{i=1}^s \pp_i^t)=\ini_<(\bigcap_{i=1}^s \pp_i^{(t)})=\ini_<(I^{(t)})\supseteq \ini_<(I^t).
\] Therefore, it follows that $\ini_<(I^{(t)})=\ini_<(I^t).$ Since $I^t\subseteq I^{(t)},$ we get $I^t= I^{(t)}.$
\end{proof}

We now investigate whether  one may use the above lemma for studying symbolic powers of  binomial edge ideals. Note that, by (\ref{intersectini}), the first condition in Lemma~\ref{inilemma} holds for any binomial edge ideal $J_G.$ In addition, as we have seen in (\ref{eqprime}), condition (a) in Lemma~\ref{inilemma} holds for any 
prime ideal $P_\MS(G)$ and any integer $t\geq 1.$

\begin{Lemma}\label{inipowers}
Let  $\MS\subset [n].$ Then $\ini_<(P_{\MS}(G)^{t})=(\ini_<(P_\MS(G)))^t,$ for every $t\geq 1.$
\end{Lemma}

\begin{proof}
To shorten the notation, we write $P$ instead of $P_{\MS}(G)$, $c$ instead of $c(\MS),$ and $J_i$ instead of $J_{\tilde{G_i}}$ for $1\leq i\leq c.$
Let $\MR(P),$ respectively $\MR(\ini_<(P))$ be the Rees algebras of $P,$ respectively $\ini_<(P).$ Then, as the sets of variables 
$\{x_j,y_j:j\in V(\tilde{G_i})\}$ are pairwise disjoint, we get
\begin{equation}\label{eqRees1}
\MR(P)=\MR((\{x_i,y_i\}_{i\in \MS}))\otimes_K (\otimes_{i=1}^c\MR(J_i)).
\end{equation}
On the other hand, since $\ini_<(P)=(\{x_i,y_i\}_{i\in \MS})+\ini_<(J_1)+\cdots+\ini_<(J_c),$ due to the fact that  
$J_1,\ldots,J_c$ are ideals in disjoint sets of variables different from $\{x_i,y_i\}_{i\in \MS}$ (see \cite{HHHKR}), we obtain
\begin{eqnarray}\label{eqRees2}
 \MR(\ini_<P)=\MR((\{x_i,y_i\}_{i\in \MS}))\otimes_K (\otimes_{i=1}^c\MR(\ini_<J_i))=\\ \nonumber
=\MR((\{x_i,y_i\}_{i\in \MS}))\otimes_K (\otimes_{i=1}^c\ini_<\MR(J_i)).
\end{eqnarray}
For the last equality we used the equality $\ini_<(J_i^t)=(\ini_<J_i)^t$ for all $t\geq 1$ which is a particular case of \cite[Theorem 2.1]{Con} and 
the equality $\MR(\ini_<J_i)=\ini_<\MR(J_i)$ due to \cite[Theorem 2.7]{CHV}.
We know that $\MR(P)$ and $\ini_<(\MR(P))$ have the same Hilbert function. On the other hand, equalities~(\ref{eqRees1}) and
 (\ref{eqRees2}) show that $\MR(P)$  and $\MR(\ini_<P)$ have the same Hilbert function since $\MR(J_i)$ and $\ini_<\MR(J_i)$ have the same Hilbert function for every $1\leq i\leq s.$ Therefore, $\MR(\ini_<P)$ and $\ini_<\MR(P)$ have the same Hilbert function.  As 
$\MR(\ini_<P)\subseteq \ini_<(\MR(P))$, we have $\MR(\ini_<P)= \ini_<(\MR(P))$, which implies by \cite[Theorem 2.7]{CHV} that 
$\ini_<(P^t)=(\ini_<P)^t$ for all $t.$
\end{proof}

\begin{Theorem}\label{iniconseq}
Let $G$ be a connected graph on the vertex set $[n].$  If $\ini_<(J_G)$ is a normally torsion-free ideal, then $J_G^{(k)}=J_G^k$ for $k\geq 1.$
\end{Theorem}

\begin{proof}
The proof is a consequence of Lemma~\ref{inipowers} combined with relations (\ref{intersectini}) and  (\ref{eqprime}).
\end{proof}

There are binomial edge ideals whose initial ideal with respect to the lexicographic order are normally torsion-free. 
For example, the binomial edge ideals which have a quadratic Gr\"obner basis  have   normally torsion-free initial ideals. They were characterized in 
\cite[Theorem 1.1]{HHHKR} and correspond to the so-called closed graphs. The graph  $G$ is  \textit{closed} if  there exists a labeling of its vertices such that for any edge $\{i,k\}$ with $i<k$ and for every $i<j<k$, we have $\{i,j\}, \{j,k\}\in E(G).$ If $G$ is closed with respect to its labeling, then, with respect to the lexicographic order $<$ on $S$ induced by the natural ordering of the indeterminates, the initial ideal of $J_G$ is $\ini_<(J_G)=(x_iy_j: i<j \text{ and }\{i,j\}\in E(G)).$ This implies that  $\ini_<(J_G)$ is the edge ideal of a bipartite graph, hence it is normally torsion-free. Therefore, we get the following. 

\begin{Corollary}\label{closed}
Let $G$ be a closed graph on the vertex set $[n].$ Then $J_G^{(k)}=J_G^k$ for $k\geq 1.$
\end{Corollary}

Let $C_4$ be the $4$-cycle with edges 
$\{1,2\},\{2,3\},\{3,4\},\{1,4\}.$ Let $<$ be the lexicographic order on $K[x_1,\ldots,x_4,y_1,\ldots,y_4]$ induced by 
$x_1>x_2>x_3>x_4>y_1>y_2>y_3>y_4.$ With respect to this monomial order, we have
\[
\ini_<(J_{C_4})=(x_1x_4y_3,x_1y_2,x_1y_4,x_2y_1y_4,x_2y_3,x_3y_4).
\]

Let $\Delta$ be the simplicial complex whose facet ideal $I(\Delta)=\ini_<(J_{C_4}).$ It is easily seen that $\Delta$ has no special odd cycle, therefore, by \cite[Theorem 10.3.16]{HH10}, it follows that $I(\Delta)$ is normally torsion-free. Note that the $4$-cycle is a complete bipartite graph, thus the equality $J_{C_4}^k=J_{C_4}^{(k)}$ for all $k\geq 1$ follows also from \cite{Oh2}.

In view of this result, one would expect that initial ideals of binomial edge ideals  of cycles are normally torsion-free. But this is not the case. Indeed, let $C_5$ be the $5$-cycle with edges $\{1,2\},\{2,3\},\{3,4\},\{4,5\},\{1,5\}$ and $I=\ini_<(J_{C_5})$  the initial ideal of $J_{C_5}$ with respect to the lexicographic order on $K[x_1\ldots,x_5,y_1,\ldots,y_5].$ By using \textsc{Singular} \cite{Soft}, we checked that $I^2\subsetneq I^{(2)}.$ Indeed, the monomial $x_1^2x_4x_5y_3y_5\in I^2$ is a minimal generator of $I^2.$ On the other hand, 
the monomial $x_1x_4x_5y_3y_5\in I^{(2)}$, thus $I^2\neq I^{(2)}$, and $I$ is not normally torsion-free. On the other hand, again with \textsc{Singular}, we have checked that $J_{C_5}^2=J_{C_5}^{(2)}.$


\begin{thebibliography}{}

\bibitem{BMS} D. Bolognini, A. Macchia, F. Strazzanti, \textit{Binomial edge ideals of bipartite graphs},  European J. Combin. \textbf{70} (2018), 1--25.
 
\bibitem  {BC03} W. Bruns, A. Conca, \textit{Gr\"obner bases and determinantal ideals}, In: Commutative Algebra, Singularities ansd Computer Algebra, J. Herzog and V. Vuletescu, Eds., NATO Science Series \textbf{115}, (2003), 9--66.

\bibitem{Con} A. Conca, \textit{Gr\"obner bases of powers of ideals of maximal minors}, J. Pure Appl. Algebra \textbf{121} (1997) 
223--231.  

\bibitem{CDeG} A. Conca, E. De Negri, E. Gorla,
\textit{Cartwright Sturmfels ideals associated to graphs and linear spaces},  J. Comb. Algebra \textbf{2}(3) (2018) 231--257.

\bibitem{CHV} A. Conca, J. Herzog, G. Valla, \textit{Sagbi bases with applications to blow--up algebras}, J. Reine Angew. Math. \textbf{474}
(1996), 113--138.

\bibitem{DEP} C. DeConcini, D. Eisenbud, C. Procesi, \textit{Young diagrams and determinantal varieties}, Invent. Math. \textbf{ 56}(2) (1980), 129--165.

\bibitem{Soft} 
W. Decker, G.-M. Greuel, G. Pfister, H.  Sch{\"o}nemann,
\newblock {\sc Singular} {4-1-2} --- {A} computer algebra system for polynomial computations.
\newblock {http://www.singular.uni-kl.de} (2019).


\bibitem{EHH} V. Ene, J. Herzog, T. Hibi, {\em Cohen-Macaulay binomial edge ideals}, Nagoya Math. J. {\bf 204} (2011),  57--68.

\bibitem{HNTT} H. T. H\`a, H. D. Nguyen, N. V. Trung, T. N. Trung, \textit{Symbolic powers of sums of ideals},  Math. Z. 
https://doi.org/10.1007/s00209-019-02323-8

\bibitem{HH10} J. Herzog, T. Hibi, {\em Monomial Ideals}, Grad. Texts in Math. \textbf{260}, Springer, 2010. 

\bibitem{HHHKR} J. Herzog, T. Hibi, F. Hreinsdotir, T. Kahle, J. Rauh, {\em Binomial edge ideals and conditional independence statements},
 Adv. Appl. Math. \textbf{45} (2010), 317--333.

\bibitem{Oh} M. Ohtani, {\em Graphs and ideals generated by some $2$-minors},  Comm. Algebra {\bf 39} (2011), no. 3,  905--917.

\bibitem{Oh2} M. Ohtani, {\em Binomial edge ideals of complete multipartite graphs}, Comm. Algebra {\bf 41} (2013), no. 10, 3858--3867.

\bibitem{RR} A. Rauf, G. Rinaldo, {\em Construction of Cohen-Macaulay binomial edge ideals},  Comm.  Algebra \textbf{42}(1) (2014), 238--252.
\bibitem{Ri} G. Rinaldo, \textit{Cohen-Macaulay binomial edge ideals of small deviation}, Bull. Math. Soc. Sci. Math. Roumanie (N.S.) 
\textbf{56} (104), No. 4 (2013), 497--503.

\bibitem{Ri2} G. Rinaldo, \textit{Cohen-Macaulay binomial edge ideals of cactus graphs}, J. Algebra Appl. \textbf{18}(4) (2019), 1950072 (18 pages).

\bibitem{Sara} S. Saeedi Madani, \textit{Binomial Edge Ideals: A Survey}, in Multigraded Algebras and Applications (V. Ene, E. Miller Eds.)
Springer Proceedings in Mathematics \& Statistics (2018), 83--94.

\end{thebibliography}
\end{document}